\documentclass[reqno]{amsart}

\usepackage{amssymb}
\usepackage{amsmath}

\newtheorem{theorem}{Theorem}[section]
\newtheorem{lemma}[theorem]{Lemma}

\theoremstyle{definition}

\theoremstyle{remark}

\newcommand{\cH}{{\mathcal{H}}}

\newcommand{\cK}{{\mathcal{K}}}

\newcommand{\fB}{{\mathfrak{B}}}

\newcommand{\fP}{{\mathfrak{P}}}
\newcommand{\fS}{{\mathfrak{S}}}
\newcommand{\bC}{{\mathbb{C}}}

\newcommand{\bR}{{\mathbb{R}}}

\newcommand{\Tr}{\mathrm{Tr}}

\newcommand{\rank}{{\mathrm{rank}}}

\title[Exposed positive maps]{Rank properties of exposed positive maps}
\author[M. Marciniak]{Marcin Marciniak}
\address{Institute of Theoretical Physics and Astrophysics, Gda{\'n}sk University,
Wita Stwosza 57, 80-952 Gda{\'n}sk, Po\-land}
\email{matmm@univ.gda.pl}
\date{February 21, 2012}
\thanks{Research supported partially by research grant of MNiSW N N202 208238 and South Africa -- Poland scientific cooperation project UID 72336.}

\begin{document}
\setlength{\baselineskip}{1.1\baselineskip}
\begin{abstract}
Let $\cK$ and $\cH$ be finite dimensional Hilbert spaces and let $\fP$ denote the cone of
all positive linear maps acting from $\fB(\cK)$ into $\fB(\cH)$. We show that each map of the form $\phi(X)=AXA^*$ or $\phi(X)=AX^\mathrm{T}A^*$ is an exposed point of $\fP$.
\end{abstract}
\maketitle


\section{Introduction}
Let us start with seting up some notation and terminology. Assume $V$ is a finite dimensional normed 
space. A subset $C$ is a {\em convex cone} (or simply {\em cone}) in $V$ if $\alpha x+\beta 
y\in C$ for all $x,y\in C$ and $\alpha,\beta\in\bR_+$. A cone $V$ is {\em pointed} if $V\cap(-V)=\{0\}$.
A convex subcone $F\subset C$ is called a {\em face} if $x-y\in F$ implies $y\in F$ for every $y\in C$.
Any face $F$ such that $F\neq\{0\}$ and $F\neq C$ will be called {\em proper}.
A proper face $F$ is said to be {\em maximal} if for any face $G$ such that $F\subset G\subset C$ we have either
$G=F$ or $G=C$. If $K\subset C$ is any subset then by $F(K)$ we will denote the smallest face containing $K$.
If $K=\{x\}$ for some $x\in C$ then we will write $F(x)$ instead of $F(\{x\})$. An element $x\in C$ will be called
{\em extremal} if $F(x)=\overline{\bR_+}x$, where $\overline{\bR_+}$ is the set of all non-negative real numbers.
The set of all extremal elements will be denoted be $\mathrm{ext}\,C$. In the sequel we will need the following
\begin{lemma}\label{l:extF}
If $C$ is a cone and $F\subset C$ is a face then $\mathrm{ext}\,F=F\cap\mathrm{ext}\,C$.
\end{lemma}
\begin{proof}
Let $x\in\mathrm{ext}\,F$. We should show that $x\in\mathrm{ext}\,C$. Assume that $y\in C$ and $x-y\in C$. Since $x\in F$ and $F$ is a face, $y\in F$ and $x-y\in F$. Now, extremality of $x$ in $F$ implies $y=\lambda x$ for some nonnegative constant $\lambda$. Thus $x\in\mathrm{ext}\,C$. We proved $\mathrm{ext}\,F\subset F\cap\mathrm{ext}\,C$. The converse inclusion is obvious.
\end{proof}

Assume now that $W$ is another finite dimensional normed space, and $V$ and $W$ are dual to each other with respect to bilinear pairing $\langle\cdot,\cdot\rangle_{\mathrm{d}}$.
For a subset $C\subset V$ (respectively $D\subset W$) we define the {\em dual cone} $C^\circ=\{y\in W:\,\mbox{$\langle x,y\rangle_{\mathrm{d}}\geq 0$ for all $x\in C$}\}$ (respectively $D^\circ=\{x\in V:\,\mbox{$\langle x,y\rangle_{\mathrm{d}}\geq 0$ for all $y\in D$}\}$). One can show that $C^{\circ\circ}$ is the smallest closed cone containing $C$.

Assume $C\subset V$ is a closed convex cone and $F\subset C$ is a face. We define
$F'=\{y\in C^\circ:\,\mbox{$\langle x,y\rangle_{\mathrm{d}}=0$ for all $x\in F$}\}$.
It is clear that $F'$ is a closed face of $C^\circ$.
We say that a face $F$ of a closed convex cone $C$ is {\em exposed} if there exists $y_0\in C^\circ$
such that $F=\{x\in C:\,\langle x,y_0\rangle_{\mathrm{d}}=0\}$.
We have the following
\begin{lemma}[\cite{Kye00}]\label{l:Kye}
Let $F$ be a closed face of a closed convex cone $C$. Then $F$ is exposed if and only if 
$F=F''$. The set $F''$ is the smallest closed face containing $F$.
\end{lemma}
An element $x\in C$ is called an {\em exposed point} of $C$ if $\overline{\bR_+}x$ is an exposed face
of $F$. The set of all exposed points of $C$ will be denoted by $\mathrm{exp}\,C$.

Now, let $\cK$ and $\cH$ be finite dimensional Hilbert spaces and $\dim\cK=m$, $\dim\cH=n$.
By $\fB(\cK)$ (respectively $\fB(\cH)$) we denote the $C^*$-algebra of all linear transformations
acting on $\cK$ (respectively $\cH$).
Let $V=\fB(\fB(\cK),\fB(\cH))$ be the space of linear mappings from $\fB(\cK)$ into $\fB(\cH)$
and let $W=\fB(\cH)\otimes\fB(\cK)$.
Assume also that some antilinear selfadjoint involutions $\cK\ni\xi\mapsto\overline{\xi}\in\cK$ and
$\cH\ni\eta\mapsto\overline{\eta}\in\cH$ are given.  
Following \cite{Sto86} (see also \cite{Kye00}) we define the following bilinear pairing between $V$ and $W$
\begin{equation}
\langle \phi,X\otimes Y\rangle_{\mathrm{d}}=\Tr\left(\phi(Y)X^{\mathrm{T}}\right)
\end{equation}
where $\phi\in V$, $X\in\fB(\cH)$, $Y\in\fB(\cK)$ and $\mathrm{T}$ is the transposition on $\fB(\cH)$ 
determined by the given antilinear selfadjoint involution on $\cH$.

Now, we choose some special convex cones $\fP\subset V$ and $\fS\subset W$. Namely $\fP$ consists of positive maps,
i.e. such maps $\phi$ that $\phi\left(\fB(\cK)_+\right)\subset\fB(\cH)_+$, while $\fS=\fB(\cH)_+\otimes\fB(\cK)_+$
(its elements are sometimes called unnormalized separable states). 
It is known that these cones are dual to each other, i.e. $\fS=\fP^\circ$ (e.g. \cite{MM01}).

The structure of extremal elements of $\fS$ is simple. For two vestors $x,y$ from a Hilbert space $\mathcal{X}$ we let $xy^*$ denote the rank $1$ operator on $X$ such that $(xy^*)z=\langle y,z\rangle x$ for $z\in\mathcal{X}$. It follows from the definition of $\fS$ that
$$\mathrm{ext}\,\fS=\{\xi\xi^*\otimes\eta\eta^*:\,\xi\in\cH,\,\eta\in\cK\}.$$

The big challenge is to describe the structure of extremal elements of the cone $\fP$. The full description
of the set $\mathrm{ext}\,\fP$ is still not done. Only some partial results are known. In \cite{Sto63} extremal elements in the convex set of unital maps acting from $\fB(\bC^2)$ into $\fB(\bC^2)$ are characterized. 
As regards the full cone $\fP$ of positive but not necessarily unital maps, it was proved in 
\cite{YH05} that every maps of the form 
\begin{equation}\label{r:dec}
\phi(X)=AXA^*\quad\mbox{or}\quad\phi(X)=AX^{\mathrm{T}}A^*,\qquad X\in\fB(\cK),
\end{equation} 
($A\in\fB(\cK,\cH)$) are extremal in $\fP$.
Moreover, several examples of non-decomposable extremal positive maps are scattered over the literature (see e.g. \cite{Cho75,Wor76,Rob85,Ha03,Bre06}).

Let us remind that due to Straszewicz's theorem (\cite{Str35}, see also \cite{Roc70}) the set $\mathrm{exp}\,\fP$ is dense in 
$\mathrm{ext}\,\fP$.
Thus, in order to do a full characterization of positive maps it is enough to describe fully the set 
of all exposed points of $\fP$. It was proved in \cite{YH05} that if $\rank A=1$ or $\rank A=m$ then 
a map of the form (\ref{r:dec}) is an exposed point of $C$. 
Recently, some new examples of exposed nondecomposable positive maps has appeared in the literature (see e.g. \cite{HK11,Chr11,ChS12a,SCh12,ChS12b}).

The aim of this paper is to provide some new examples of exposed positive maps. We will use results concerning rank properties of extremal positive maps described in \cite{Mar09}.

\section{Main result}
Now we are ready to formulate our main theorem.
\begin{theorem}\label{t}
Every map of the form (\ref{r:dec}) is an exposed point of $C$.
\end{theorem}

In the proof of the above theorem we will need more or less known two lemmas.
For any $A\in\fB(\cK,\cH)$ we denote $\Vert A\Vert_2=\left(\Tr(A^*A)\right)^{1/2}$.
\begin{lemma}\label{l:1to1}
There is a unique linear map $\fB(\cK,\cH)\ni A\mapsto f_A \in(\cH\otimes\cK)^*$ 
such that $f_A(\xi\otimes\eta)=\langle\overline{\xi},A\eta\rangle$ for any $\xi\in\cH$ and $\eta\in\cK$.
Moreover, we have $\Vert f_A\Vert=\Vert A\Vert_2$ for $A\in\fB(\cK,\cH)$.
\end{lemma}
\begin{proof}
Observe that $\cH\times \cK\ni(\xi,\eta)\mapsto\langle\overline{\xi},A\eta\rangle\in\bC$ is a bilinear
form. It follows from the universality property of a tensor product that this form has a unique
lift to a linear functional $f_A$ on the $\cH\otimes\cK$.
Now, if $f\in (\cH\otimes\cK)^*$ then we define $\varphi(\xi,\eta)=f(\overline{\xi}\otimes\eta)$ 
for $\xi\in\cH$ and $\eta\in\cK$. It is a sesquilinear form on $\cH\times\cK$, so there is 
$A\in\fB(\cK,\cH)$ such that $\varphi(\xi,\eta)=\langle\xi,A\eta\rangle$. Hence we have
$f(\xi\otimes\eta)=\varphi(\overline{\xi},\eta)=\langle\overline{\xi},A\eta\rangle=f_A(\xi\otimes\eta)$.

It remains to show the equality of norms. Let $\eta_1,\ldots,\eta_m$ be an orthonormal basis of $\cK$ and 
$\xi_1,\ldots,\xi_m$ be any system of vectors from $\cH$. Observe that 
$\left\Vert\sum_i\xi_i\otimes\eta_i\right\Vert^2=\sum_i\Vert\xi_i\Vert^2$ and $\Vert A\Vert_2^2=\sum_i\Vert A\eta_i\Vert^2$.
Thus
\begin{eqnarray*}
\left|f_A\left(\sum_i\xi_i\otimes\eta_i\right)\right|
&=&
\left|\sum_i\langle\overline{\xi_i},A\eta_i\rangle\right|\leq
\sum_i\left\Vert\overline{\xi_i}\right\Vert\Vert A\eta_i\Vert\leq \\
&\leq &
\left(\sum_i\Vert\xi_i\Vert^2\right)^{1/2}\left(\sum_i\Vert A\eta_i\Vert^2\right)^{1/2}= 
\left\Vert\sum_i\xi_i\otimes\eta_i\right\Vert\Vert A\Vert_2.
\end{eqnarray*}
Hence $\Vert f_A\Vert\leq\Vert A\Vert_2$. Now, observe that 
$\left\Vert\sum_i\overline{A\eta_i}\otimes\eta_i\right\Vert=\left(\sum_i\left\Vert A\eta_i\right\Vert^2\right)^{1/2}=\Vert A\Vert_2$ and 
$\left|f_A\left(\sum_i\overline{A\eta_i}\otimes\eta_i\right)\right|=\Vert A\Vert_2^2=\Vert A\Vert_2\left\Vert\sum_i\overline{A\eta_i}\otimes\eta_i\right\Vert$. Thus we conclude $\Vert f_A\Vert=\Vert A\Vert_2$.
\end{proof}
\begin{lemma}\label{l:my}
Let $A\in\fB(\cK,\cH)$ and $\rank A\geq 2$. Assume that an operator $B\in\fB(\cK,\cH)$ 
satisfies the following condition: for any $\xi\in\cH$ and $\eta\in\cK$ if $\langle\xi,A\eta\rangle=0$ then
$\langle\xi,B\overline{\eta}\rangle=0$. Then $B=0$.
\end{lemma}
\begin{proof}
Let $\eta_1,\eta_2,\ldots,\eta_m$ be such that $\overline{\eta_1},\overline{\eta_2},\ldots,\overline{\eta_m}$ form 
an orthonormal system of eigenvectors of the operator $A^*A$.
Let $r=\rank A$. Thus we may assume that $\overline{\eta_{r+1}},\ldots,\overline{\eta_m}$ correspond to zero 
eigenvalue while
$\overline{\eta_1},\ldots,\overline{\eta_r}$ correspond to non-zero eigenvalues. Thus vectors $A\overline{\eta_1},\ldots,A\overline{\eta_r}$ span the image of the operator $A$.
Take any $j\in\{r+1,\ldots,m\}$.
Observe that $A\overline{\eta_j}=0$, so $\langle\xi,A\overline{\eta_j}\rangle=0$ for each $\xi\in\cH$. 
It follows from the assumption of the
Lemma that $\langle\xi,B\eta_j\rangle=0$ for any $\xi\in\cH$, so $B\eta_j=0$.
Now, let $j\in\{1,\ldots,r\}$. It follows from the assumption that for any $\xi\in A\cK^\bot$ we have $\langle\xi,B\eta_j\rangle=0$. Hence $B\eta_j\in A\cK$.
Now consider also some $k\in\{1,\ldots,r\}$ such that $k\neq j$. For every $z\in\bC$ define 
$$\zeta_z=-\overline{z}\Vert A\overline{\eta_k}\Vert^2 A\overline{\eta_j}+\Vert A\overline{\eta_j}\Vert^2 A\overline{\eta_k},\quad
\rho_z=\overline{\eta_j}+z\overline{\eta_k}.$$
Observe that $\langle\zeta_z,A\rho_z\rangle=0$. On the other hand we have
\begin{eqnarray*}
\langle\zeta_z,B\overline{\rho_z}\rangle
&=&-z\Vert A\overline{\eta_k}\Vert^2\langle A\overline{\eta_j},B\eta_j\rangle-|z|^2\Vert A\overline{\eta_k}\Vert^2\langle A\overline{\eta_j},B\eta_k\rangle+\\
&&+\Vert A\overline{\eta_j}\Vert^2\langle A\overline{\eta_k},B\eta_j\rangle+
\overline{z}\Vert A\eta_j\Vert^2\langle A\overline{\eta_k},B\eta_k\rangle.
\end{eqnarray*}
It follows from the assumption that this expression is equal to zero for every $z\in\bC$. 
Thus we conclude that $\langle A\overline{\eta_k},B\eta_j\rangle=0$ for any $k=1,\ldots,r$.
Since $B\eta_j\in A\cK$ and $A\overline{\eta_1},\ldots,A\overline{\eta_r}$ span $A\cK$, we conclude
that $B\eta_j=0$.

Thus we proved that $B\eta_j=0$ for any $j=1,2,\ldots,m$. The Lemma follows from the fact that $\eta_1,\ldots,\eta_m$ is a basis of $\cK$.
\end{proof}

\begin{proof}[Proof of Theorem \ref{t}]
Let us consider a map $\phi\in C$. 
It follows from Lemma \ref{l:Kye} that $\phi$ is an exposed point if and only if $\{\phi\}''=\overline{\bR_+}\phi$.
Let us calculate the face $\{\phi\}'\subset D$ firstly. Since any closed convex cone in finite dimensional space $W$
is a closed convex hull of its extremal elements, in order to determine the face $\{\phi\}'$ it is enough to describe
its extremal elements. They are those elements of $\{\phi\}'$ which are extremal in $D$ (cf. Lemma \ref{l:extF}).
Thus, one should find all pairs $(\xi,\eta)\in\cH\times\cK$ such that $\xi\xi^*\otimes\eta\eta^*\in\{\phi\}'$.
An easy calculation shows that this holds if and only if
$\langle\overline{\xi},\phi(\eta\eta^*)\overline{\xi}\rangle=0$ or, 
equivalently, $\phi(\eta\eta^*)\overline{\xi}=0$. As a consequence we get the following characterization
of the face $\{\phi\}''$: if $\psi\in C$, then $\psi\in\{\phi\}''$ if and only if $\psi(\eta\eta^*)\overline{\xi}=0$
for all pairs $(\xi,\eta)\in\cH\times\cK$ such that $\phi(\eta\eta^*)\overline{\xi}=0$.

Now, assume that $\phi(X)=AXA^*$ where $A$ is some linear map from $\cK$ into $\cH$. 
One can easily show that $\xi\xi^*\otimes\eta\eta^*\in\{\phi\}'$ if and only if $\langle\overline{\xi},A\eta\rangle=0$. We will show that for $\psi\in C$
if $\psi\in\{\phi\}''$ then $\psi$ is rank 1 non-increasing in the sense of \cite{Mar09}. 
Let $\eta\in\cK$. Consider any $\xi\in\cH$ such that
$\xi\bot A\eta$. Then $\overline{\xi}\overline{\xi}^*\otimes\eta\eta^*\in\{\phi\}'$ what is equivalent to
$\phi(\eta\eta^*)\xi=0$. So, it follows from the preceding paragraph that $\psi(\eta\eta^*)\xi=0$ for any $\xi\in\{A\eta\}^\bot$. Thus $\psi(\eta\eta^*)$ is a non-negative multiple of rank 1 positive element $(A\eta)(A\eta)^*$.

Now, it follows from Theorem 2.2 in \cite{Mar09} that we have three possibilities:
\begin{enumerate}
\item[(i)] $\psi(X)=\omega(X)Q$ for some positive functional $\omega$ on $\fB(\cK)$ and a $1$-dimensional projection $Q$ on $\cH$,
\item[(ii)] $\psi(X)=BXB^*$ for some $B\in\fB(\cK,\cH)$,
\item[(iii)] $\psi(X)=BX^{\mathrm{T}}B^*$ for some $B\in\fB(\cK,\cH)$.
\end{enumerate}
Assume firstly that $\psi$ satisfies the condition (ii), i.e. $\psi(X)=BXB^*$ for some $B\in\fB(\cK,\cH)$.
From the above considerations we conclude that $B$ satisfies the following condition: for any $(\xi,\eta)\in\cH\times\cK$, if $\langle\overline{\xi},A\eta\rangle=0$ then $\langle\overline{\xi},B\eta\rangle=0$.
Now apply Lemma \ref{l:1to1}. Using notations introduced there we can write the above condition as 
$\ker f_A\subset\ker f_B$. It is equivalent to the fact that $f_B$ is a multiple of $f_A$. Hence $B=\lambda A$
for some $\lambda\in\bC$, and $\psi=|\lambda|^2\phi$, and consequently $\psi\in\overline{\bR_+}\phi$.

Secondly, consider the case (i), i.e. let $\psi(X)=\omega(X)Q$, where $\omega$ is some positive functional on $\fB(\cK)$
and $Q$ is a $1$-dimensional projection on $\cH$.
Let $\omega(X)=\Tr(RX)$ where $R$ is some positive operator on $\cK$ and let $Q=\zeta\zeta^*$ for some non-zero $\zeta\in\cH$.
We will show that $\rank R\leq 1$. 
To this end observe that the condition $\psi\in\{\phi\}''$ is equivalent to the following:
for any $(\xi,\eta)\in\cH\times\cK$, if $\langle\xi,A\eta\rangle=0$ then $\langle\eta,R\eta\rangle=0$ or $\langle\zeta,\xi\rangle=0$.
It follows that $\langle\eta,R\eta\rangle=0$ provided that $A\eta=0$, hence $\ker A\subset\ker R$.
Assume that there are two vectors $\eta_1,\eta_2\in\cK$ such that $R\eta_1,R\eta_2$ are linearly independent.
Then $A\eta_1,A\eta_2$ are also linearly independent. 
Fix $i\in\{1,2\}$. Since $\psi\in\{\phi\}''$,
for any $\xi\in\{A\eta_i\}^\bot$ we have $\langle\eta_i,R\eta_i\rangle=0$ or $\langle\zeta,\xi\rangle\zeta=0$. We assumed $\langle\eta_i,R\eta_i\rangle\neq 0$, so we conclude that $\zeta\bot\xi$. Thus we proved
that $\zeta\in\bC A\eta_i$ for $i=1,2$. But $A\eta_1,A\eta_2$ are independent, so $\zeta=0$ which is a contradiction
to the assumption. Thus we proved that $R=\rho\rho^*$ for some $\rho\in\cK$. Now, we can write 
$\phi(X)=\zeta\rho^*X\rho\zeta^*=\zeta\rho^*X(\zeta\rho^*)^*$ and we arrived at the previously described case (ii).

Finally, assume that $\phi(X)=CX^{\mathrm{T}}C^*$ for some $C\in\fB(\cK,\cH)$.
Observe that in this case the condition $\psi(\eta\eta^*)\overline{\xi}=0$ is equivalent to $\langle\overline{\xi},C\overline{\eta}\rangle=0$. Thus $\psi\in\{\phi\}''$ if and only if 
for any $(\xi,\eta)\in\cH\times\cK$, $\langle\overline{\xi},A\eta\rangle=0$ implies $\langle\overline{\xi},C\overline{\eta}\rangle=0$. It follows from Lemma \ref{l:my}
that if $\rank A\geq 2$ then $C=0$. It remains to consider the case when $\rank A\leq 1$. If $A\eta=0$ then $C\overline{\eta}=0$, hence $\overline{\ker A}\subset\ker C$, and consequently $\rank C\leq 1$.
Then $C=\zeta\rho^*$ for some $\zeta\in\cH$ and $\rho\in\cK$. Hence $\psi(X)=\langle\rho,X^{\mathrm{T}}\rho\rangle\zeta\zeta^*$. Since $X\mapsto\langle\rho,X^{\mathrm{T}}\rho\rangle$
is a positive functional on $\fB(\cK)$, we came to the case (i).
This ends the proof for $\phi(X)=AXA^*$.

If $\phi(X)=AX^{\mathrm{T}}A^*$ then the proof is similar. To show that any $\psi\in\{\phi\}''$ is a multiple
of $\phi$ one should firstly show that $\psi$ is rank 1 non-increasing, then consider cases (iii), (i) and (ii).
\end{proof}

\end{document}